\ifx\shlhetal\undefinedcontrolsequence\let\shlhetal\relax\fi

\documentclass[11pt]{amsart}

\usepackage{amsmath}
\usepackage{amssymb}

\newtheorem{theorem}{Theorem}[section]

\newtheorem{lemma}[theorem]{Lemma}

\newtheorem{proposition}[theorem]{Proposition}
\newtheorem{corollary}[theorem]{Corollary}

\newtheorem{conjecture}[theorem]{Conjecture}

\theoremstyle{definition}
\newtheorem{definition}[theorem]{Definition}

\theoremstyle{remark}

\newcount\skewfactor
\def\mathunderaccent#1#2 {\let\theaccent#1\skewfactor#2
\mathpalette\putaccentunder}
\def\putaccentunder#1#2{\oalign{$#1#2$\crcr\hidewidth
\vbox to.2ex{\hbox{$#1\skew\skewfactor\theaccent{}$}\vss}\hidewidth}}

\def\smallbox#1{\leavevmode\thinspace\hbox{\vrule\vtop{\vbox
   {\hrule\kern1pt\hbox{\vphantom{\tt/}\thinspace{\tt#1}\thinspace}}
   \kern1pt\hrule}\vrule}\thinspace}

\title{The congruence subgroup problem for the free metabelian group on two generators}
\author{David El-Chai Ben-Ezra}
\address{Institute of Mathematics
 The Hebrew University of Jerusalem
 Jerusalem 91904, Israel}
\email{davidel-chai.ben-ezra@mail.huji.ac.il}
\thanks{ I offer my deepest thanks to my great teacher Prof. Alexander Lubotzky for his  devoted and sensible guidance, and for Mr. Baum and Mr. Schwartz, Ruding foundation trustees, for their generous support. I would also like to thank the referees for several critical remarks which led to some significant simlifications.}

\begin{document}
\let\labeloriginal\label
\let\reforiginal\ref

\begin{abstract}
In this paper we  describe the profinite completion of the free solvable group on $m$ generators of solvability length $r \ge 1$. Then, we show that for $m=r=2$, the free metabelian group on two generators does not have the Congruence Subgroup Property.

\end{abstract}

\maketitle


\section{Introduction}

Let $\Gamma:=GL_m(\mathbb{Z})$ be the matrix group of the $m \times m$ invertible matrices over $\mathbb{Z}$. For $n \in \mathbb{N}$, set: $\Gamma[n]:=\ker (\Gamma \to  GL_m(\mathbb{Z}_n))$ where $\mathbb{Z}_n:=\mathbb{Z}/n\mathbb{Z}$.
 The classiscal congruence subgroup problem asks whether every subgroup of $\Gamma$ of finite index contains a congruence subgroup $\Gamma[n]$, for some $n \in \mathbb{N}$?
\newline

By referring to  $GL_m(\mathbb{Z})$ as the automorphism group of $\mathbb{Z}^m$, one can generalize the congruence subgroup problem as follows: Let $G$ be a finitely generated group, and let $\Gamma:=Aut(G)$ be the automorphism group of $G$. For a characteristic subgroup $K$ of finite index in $G$, set $\Gamma[K]:=\ker (\Gamma \to Aut(G/K))$. Now, the generalized congruence subgroup problem asks: does every finite index subgroup of $\Gamma$ contain a congruence subgroup of the form $\Gamma[K]$ for some characteristic subgroup $K$ of finite index in $G$? In the case where the answer to this question is positive we say that $G$ has the CSP (Congruence Subgroup Property).
\newline

We prefer to say that $G$ has the CSP rather than $\Gamma$, as there could potentially exist two groups $G_1$, $G_2$ with $Aut(G_1)$, $Aut(G_2)$ isomorphic such that just one of these automorphism groups has the CSP. But when there is no danger of ambiguities we will also say that $\Gamma=Aut(G)$ has the CSP.\newline

It is well known that $\mathbb{Z}^2$, the free abelian group on two generators, does not have the CSP. This result was known already in the 19th century ([Su], [Ra]). On the other hand, in 2001 Asada [A] showed, by using methods of algebraic geometry,  that $F_2$, the free group on two generators does have the CSP. Later (2011), Bux, Ershov and Rapinchuk [BER] provided a group theoretic version of this result and proof.\newline

Our paper deals with an intermediate case: the free metabelian group on two generators (or more generally, the free solvable group of length $r \ge 2$ on two generators). As we will see, a large part of Bux-Ershov-Rapinchuk's proof works for the automorphism group of the free metabelian group, which initially caused us to believe that it has the CSP. However, as we will show below, it does not.\newline

\begin{theorem}
\label{main} The free metabelian group on two generators $\Phi$ does not have the CSP.
\end{theorem}

Here is the main line of the proof: the generalized congruence subgroup problem asks whether the profinite topology of $Aut(G)$ is the same as the congruence topology. By some standard arguments from [BER], one can pass to $Out(G)$ (if $G$ is residually finite and $\hat{G}$ is centerless, which is the case for our groups - see Theorem \ref{residual} and proposition \ref{center}). In $\S 2$ we will give an explicit description of the profinite topology and completion of the free solvable group $G=\Phi_{m,r}$ on $m$ generators and solvability length $r$. This enables one to describe the congruence topology of $Aut(G)$. What makes now the case $m=2$ so special is that $Out(\Phi_{2,r}) \cong GL_2 (\mathbb{Z})$, and the congruence subgroup problem for $\Phi_{2,r}$ boils down to the question as to whether the congruence topology of $Out(\Phi_{2,r})$ induces the full profinite topology  of $GL_2 (\mathbb{Z})$. Note that for $F_2$, the free group on two generators, this is indeed the case: the congruence topology of $Out(F_2)$ is the full profinite topology of $Out(F_2) \cong GL_2 (\mathbb{Z})$.\newline

On the other hand, we will show that for $\Phi:=\Phi_{2,2}$, the free metabelian group on two generators, the congruence topology of $Out(\Phi) \cong GL_2 (\mathbb{Z})$ is equal to the classical congruence topology of $GL_2 (\mathbb{Z})$. As it was pointed out before, it is well known that this topology is much weaker than the profinite topology of $GL_2 (\mathbb{Z})$. Hence $\Phi$ fails to have the CSP.\newline

Our method and result suggest conjecturing that all the free solvable groups $\Phi_{2,r}$ (on two generators) do not have the CSP.\newline

The paper is organized as follows: in $\S1$ we present the Romanovskii embedding, a generalization of the well-known Magnus embedding, and use it in $\S2$ to describe the profinite completion of a finitely generated free solvable group. In $\S3$ we quote a result of Bux, Ershov and Rapinchuk which they used to prove that $F_2$ does  have the CSP. In $\S4$ we formulate a conjecture on $\Phi_{2,r}$ that would imply that $\Phi_{2,r}$ does not have the CSP. Finally, in $\S5$ we prove this conjecture for $\Phi:=\Phi_{2,2}$, and conclude that $\Phi$ does not have the CSP.

\section{Romanovskii embedding}

Let $F$ be the free group on $m$ generators, and let $R$ be a normal subgroup of $F$. Denote by $R'$ the commutator subgroup of $R$. Following the well-known Magnus embedding (see [B], [RS], [M]) and its generalization by Shmel'kin (see [Sh1, Sh2]), Romanovskii found a faithful embedding of $F/(R'R^n),\: n \in \mathbb{N} \cup \{0\}$ (with the notation $R^0=1$), into a matrix group, which enables to describe $F/(R'R^n)$ through $F/R$. In this section we will present the Romanovskii embedding, and as mentioned, we will use it in the next section to describe the profinite completion of a finitely generated free solvable group.\newline

Define $G:=F/R$ and $\tilde{G}_n:=F/R'R^n$, $n \in \mathbb{N} \cup \{0\}$. Denote by $\{f_1,\ldots,f_m\}$, $\{g_1,\ldots,g_m\}$ and $\{\tilde{g}_1,\ldots,\tilde{g}_m\}$ the appropiate generators of $F$, $G$ and $\tilde{G}_n$ respectively. By the mapping $f_i \mapsto \tilde{g}_i \mapsto g_i$, $i=1,\ldots ,m$, we get a natural epimorphism $F \to \tilde{G}_n \to G$. For $n \in \mathbb{N} \cup \{0\}$ denote $\mathbb{Z}_n=\mathbb{Z}/n\mathbb{Z}$. Now, for the group $G$, denote by $\mathbb{Z}_n [G]$ the group ring of $G$ over $\mathbb{Z}_n$, and by $T_n:=\mathbb{Z}_n [G]t_{1} +\ldots +\mathbb{Z}_n [G]t_{m} $ the left free $\mathbb{Z}_n [G]$ module with basis $\{t_{1} ,t_{2} ,\ldots ,t_{m}\} $. Now, define: \[R_n (G):=\left\{\left(\begin{array}{cc} {g} & {t} \\ {0} & {1} \end{array}\right)_{} \left|_{} g\in G,_{} t\in T_n \right. \right\}.\]
One can notice that $R_n (G)$ has a group structure under formal matrix multiplication given by the formula:
\[\left(\begin{array}{cc} {g} & {t} \\ {0} & {1} \end{array}\right)\cdot \left(\begin{array}{cc} {h} & {s} \\ {0} & {1} \end{array}\right):=\left(\begin{array}{cc} {gh} & {gs+t} \\ {0} & {1} \end{array}\right).\]

\begin{theorem} [\text{[Rom]}, Theorem 2]
\label{magnus}
\textbf{}
\begin{enumerate}
\item By the mapping $\alpha$, defined by: $\tilde{g}_{i} \mapsto \left(\begin{array}{cc} {g_{i} } & {t_{i} } \\ {0} & {1} \end{array}\right)$ for  $i=1,\ldots ,m$, we get a faithful representation of $\tilde{G}_n$ as a subgroup of $R_n (G)$.
\item An element of the form $\left(\begin{array}{cc} {g} & {\sum _{i=1}^{m} a_{i} t_{i} } \\ {0} & {1} \end{array}\right)\in R_n (G)$ belongs to $Im(\alpha) $ if and only if $1-g=\sum _{i=1}^{m} a_i (1-g_i)$. In particular,
\[\left(\begin{array}{cc} {1} & {\sum _{i=1}^{m} a_{i} t_{i} } \\ {0} & {1} \end{array}\right) \in Im(\alpha) \iff \sum _{i=1}^{m} a_i (1-g_i)=0.\]
\end{enumerate}

\end{theorem}

Considering the formula for multiplication in $R_n (G)$, it is easy to see that under $\alpha$, the subgroup $R/R'R^n \leqslant \tilde{G}_n$ is isomorphic to $\left\{\left(\begin{array}{cc} {1} & {{t}} \\ {0} & {1} \end{array}\right) \in Im(\alpha) \right\}$.

\textbf{}

\section{The profinite completion of a finitely generated free solvable group}

As before, let $F$ be the free group on $m$ generators. Denote by $F^{(1)} :=[F,F]$ the commutator subgroup of $F$, and by induction, define: $F^{(r+1)} :=[F^{(r)} ,F^{(r)} ]$. Denote by $\Phi _{r} :=F/ F^{(r)}$ the free solvable group on $m$ generators of solvability length $r$, and define $\Phi _{0} :=F/F =\left\{e\right\}$. To simplify notations, we write $F$, $\Phi_r$, $\Phi_{r,n}$ instead of $F_m$, $\Phi_{m,r}$, $\Phi_{m,r,n}$ respectively, as the results of this section hold for every finite number of generatos bigger than $1$. Sometime we will use also the notation $\Phi _{r,0}=\Phi _{r}$. In this section we will present, using the Romanovskii embedding, an "explicit" description of $\hat{\Phi }_{r} $, in the sense that we will find finite quotients $\left\{\Phi _{r,n} \right\}_{n=1}^{\infty } $, of $\Phi _{r} $, such that $\hat{\Phi }_{r} \cong \mathop{\underleftarrow{\lim } }\limits_{n\in {\rm N} } \Phi _{r,n} $. This explicit description will help us later to describe some properties of $\hat{\Phi }_{r} $, and to prove the main theorem.\newline

\begin{definition}
For $n\in \mathbb{N}$, define by induction on $r$ the following groups: for $r=0$, $K_{0,n} :=F $ and for $r+1$, $K_{r+1,n}: =K'_{r,n} (K_{r,n} )^{n} $. Denote $\Phi _{r,n} :=F/ K_{r,n}  $.
\end{definition}

\begin{proposition}
\label{finite}
For $n\in \mathbb{N}$ we have:
for $r=0$, $\left|\Phi _{0,n} \right|=1$, and for $r>0$, $\left|\Phi _{r+1,n} \right|=\left|\Phi _{r,n} \right|\cdot n^{\left|\Phi _{r,n} \right|\cdot (m-1)+1} $.
\end{proposition}

\begin{proof}
For $r=0$ the result is obvious. For $r+1$: by a well-known theorem of Schreier, we know that $K_{r,n} $ is free, and its rank is:  $\left|\Phi _{r,n} \right|\cdot (m-1)+1$. Therefore, $K_{r,n}/ K'_{r,n} (K_{r,n} )^{n}  $ isomorphic to $\left(\mathbb{Z} _{n} \right)^{\left|\Phi _{r,n} \right|\cdot (m-1)+1} $ and thus, $\left|\Phi _{r+1,n} \right|=\left|\Phi _{r,n} \right|\cdot n^{\left|\Phi _{r,n} \right|\cdot (m-1)+1} $.

\end{proof}

\begin{proposition}
\label{ab}
For $n\in \mathbb{N}$, denote $\Phi _{r,n}^{ab}=\Phi _{r,n}/[\Phi _{r,n},\Phi _{r,n}]$. Then $\Phi _{r,n}^{ab}=(\mathbb{Z}/ {n^r}\mathbb{Z})^m$.
\end{proposition}

\begin{proof}
We need to show that:
\[[F,F] \cdot K_{r,n}=[F,F] \cdot F^{(n^r)}\]

By induction on $r$. For $r=1$ the result is obvious. For $r+1$:
\[[F,F] \cdot K_{r+1,n}=[F,F]\cdot [K_{r,n},K_{r,n}]\cdot (K_{r,n})^n=  [F,F] \cdot (F^{(n^r)})^n=[F,F] \cdot F^{(n^{r+1})}\]

\end{proof}

The next proposition follows easily from the definitions.

\begin{proposition}
\label{char}
For $n\in \mathbb{N}$, denote $L_{r,n} :=\ker (\Phi _{r} \to \Phi _{r,n} )$. Then, $L_{r,n} $ is a characteristic subgroup of $\Phi _{r} $.
\end{proposition}

We need now few notations:

\begin{definition}
For $r,n\in \mathbb{N} \cup \{0\} $ denote:

\begin{enumerate}
\item[-] $x_{r,n,1} ,\ldots ,x_{r,n,m} $ to be the $m$ standard generators of $\Phi _{r,n} $.

\item[-] $\mathbb{Z} _{n} [\Phi _{r,n} ]$ - the group ring of $\Phi _{r,n} $ over $\mathbb{Z} _{n}:=\mathbb{Z}/n\mathbb{Z}$.

\item[-] ${T}_{r,n} :=\mathbb{Z}_{n} [\Phi _{r,n} ]{t}_{r,n,1} +\ldots +\mathbb{Z}_{n} [\Phi _{r,n} ]{t}_{r,n,m} $, will be the $\mathbb{Z}_{n} [\Phi _{r,n} ]$ left free module with the base ${t}_{r,n,1} ,{t}_{r,n,2} ,\ldots ,{t}_{r,n,m} $. \newline

\item[-]  $R (\Phi_{r,n}) :=\left\{\left(\begin{array}{cc} {g} & {t} \\ {0} & {1} \end{array}\right)_{} \left|_{} g\in \Phi _{r,n} ,_{} t\in T_{r,n} \right. \right\}$.

\end{enumerate}
\end{definition}

The next proposition is a special case of Theorem \ref{magnus}.

\begin{proposition}
\label{base}
For $r,n\in  \mathbb{N} $, the correspondence: $x_{r+1,n,i} \leftrightarrow \left(\begin{array}{cc} {x_{r,n,i} } & {{t}_{r,n,i} } \\ {0} & {1} \end{array}\right)$ for $i=1,\ldots ,m$, gives a faithful representation of $\Phi _{r+1,n} $ as a subgroup of $R (\Phi_{r,n}) $.
\end{proposition}

\begin{proposition}
\label{comp}
Denote by $\pi _{r,n} $ the natural homomorphism from $\Phi _{r} $ to $\Phi _{r,n} $. Then: every finite quotient $\pi :\Phi _{r} \to G$ whose order divides $n$ factorized through $\Phi _{r,n} $. I.e. there is $\tilde{\pi }$, $\Phi _{r} \stackrel{\pi _{r,n} }{\longrightarrow} \Phi _{r,n} \stackrel{\tilde{\pi }}{\longrightarrow} G$ such that: $\tilde{\pi }\circ \pi _{r,n} =\pi $. In particular, $\hat{\Phi }_{r} \cong \mathop{\underleftarrow{\lim } }\limits_{n\in {\rm N} } \Phi _{r,n}$.
\end{proposition}

\begin{proof}
Let $G$ be a quotient of $\Phi _{r}$  whose order divides $n$. Set $G=F/R$ where $R\geq F^{(r)}$ and $R\geq F^n$ -  i.e. $R\geq F^{(r)} \cdot  F^n$. We recall that $\Phi _{r,n}=F/K_{r,n}$. According to these notations we actually need to prove that $ R\geq K_{r,n}$. We will show by induction on $l$ that for every $l \in \mathbb{N}$ we have $R \cdot F^{(l)}\geq K_{l,n}$. Therefore, as $R\geq F^{(r)}$ we will deduce that in particular, $R\geq K_{r,n}$ as required.\newline

For $l=1$ we have $K_{1,n}=F'F^n$ and in this case we have:
\[K_{1,n} = F^n \cdot F'  \leq R \cdot F^{(1)}\]

For $l+1$:
\[K_{l+1,n}=[K_{l,n},K_{l,n}] \cdot (K_{l,n})^n \leq  [R \cdot F^{(l)},R \cdot F^{(l)}] \cdot (R \cdot F^{(l)})^n \leq (R \cdot F^{(l+1)}) \cdot R=R \cdot F^{(l+1)} \]

\end{proof}

The last proposition gives us an "explicit" description of the structure of $\hat{\Phi }_{r} $. For example:

\begin{corollary}
Let  $z_{n,1}, \dots, z_{n,m}$ be the standard generators of $(\mathbb{Z}_n)^m$, then:
\begin{enumerate}
\item[-] $\hat{\Phi }_{1} $$ \cong \mathop{\underleftarrow{\lim } }\limits_{n\in {\mathbb{N}} }\langle z_{n,1}, \dots, z_{n,m}  \rangle$.
\item[-] $\hat{\Phi }_{2} $$ \cong \mathop{\underleftarrow{\lim } }\limits_{n\in {\mathbb{N}} } \left\langle \left( \begin{array}{cc} {z_{n,1} } & {{t}_{1,n,1} } \\ {0} & {1} \end{array}\right) , \dots, \left( \begin{array}{cc} {z_{n,m} } & {{t}_{1,n,m} } \\ {0} & {1} \end{array}\right) \right\rangle $.
\item[-] $\hat{\Phi }_{3} $$ \cong \mathop{\underleftarrow{\lim } }\limits_{n\in {\mathbb{N}} } \left\langle\left( \begin{array}{cc} {\left( \begin{array}{cc} {z_{n,1} } & {{t}_{1,n,1} } \\ {0} & {1} \end{array}\right) } & {{t}_{2,n,1} } \\ {0} & {1} \end{array}\right) , \dots, \left( \begin{array}{cc} {\left( \begin{array}{cc} {z_{n,m} } & {{t}_{1,n,m} } \\ {0} & {1} \end{array}\right) } & {{t}_{2,n,m} } \\ {0} & {1} \end{array}\right) \right\rangle$
\end{enumerate}
\end{corollary}

By using this "explicit" description of $\hat{\Phi }_{r} $, we will show a few properties of $\hat{\Phi }_{r} $, which will help us later solve the congruence subgroup problem for $\Phi _{2,2}$, the free metabelian group on two generators.\newline

\begin{proposition}
\label{order}
Let $n \in \mathbb{N}$ and let  $x_{r,n,i} $ be one of the generators of $\Phi _{r,n} $. Then, $O(x_{r,n,i} )$, the order of $x_{r,n,i} $, is equals $n^{r} $.
\end{proposition}

\begin{proof}
By induction on $r$. For $r=0$, the result is obvious. For  $r+1$: by proposition \ref{base}, one can write: $x_{r+1,n,i} =\left(\begin{array}{cc} {x_{r,n,i} } & {{t}_{r,n,i} } \\ {0} & {1} \end{array}\right)$. By the induction hypothesis: $O(x_{r,n,i} )=n^{r} $. For $0<k\in \mathbb{N} $ we can uniquely write: $k=k_{1} n^{r} +k_{2} $, when: $k_{1} \in  \mathbb{N} \cup \{0\}  $ and $1\le k_{2} \le n^{r} $. Thus, simple calculation shows that for $0<k\in \mathbb{N} $:
\[\left(x_{r+1,n,i} \right)^{k} =\left(\begin{array}{cc} {x_{r,n,i} } & {{t}_{r,n,i} } \\ {0} & {1} \end{array}\right)^{k} =\]
\[=\left(\begin{array}{cc} {\left(x_{r,n,i} \right)^{k_{1} n^{r} +k_{2} } } & {\left(\sum _{j=0}^{k_{1} n^{r} +k_{2} -1}\left(x_{r,n,i} \right)^{j}  \right) {t}_{r,n,i} } \\ {0} & {1} \end{array}\right)=\]
\[=\left(\begin{array}{cc} {\left(x_{r,n,i} \right)^{k_{2} } } & {\left(1+\ldots +\left(x_{r,n,i} \right)^{k_{2} -1} +k_{1} \left(1+\ldots +\left(x_{r,n,i} \right)^{n^{r} -1} \right)\right) {t}_{r,n,i} } \\ {0} & {1} \end{array}\right) .\]

Looking at the left upper coordinate, one can see that by the induction hypothesis $k_{2} =n^{r} $ is a necessary condition for $x_{r+1,n,i}$ to satisfy $\left(x_{r+1,n,i} \right)^{k} =e$. Therefore, if we want $x_{r+1,n,i}$ to satisfy $\left(x_{r+1,n,i} \right)^{k} =e$, we need that the right upper coordinate will equal $0$ under the condition $k_{2} =n^{r} $, i.e. $(k_{1} +1)\left(1+\ldots +\left(x_{r,n,i} \right)^{n^{r} -1} \right)=0$. The right upper coordinate belongs to ${T}_{r,n} $ which is a free module over $\mathbb{Z} _{n} [\Phi _{r,n} ]$, and therefore, it happens for the first time when $k_{1} =n-1$, i.e. $k=(n-1)n^{r} +n^{r} =n^{r+1} $.

\end{proof}

\begin{proposition}
\label{center}
Let $\hat{\Phi }_{r} $ be the free profinite solvable group on $m\ge 2$ generators of solvability length $r\ge 2$. Then, $Z(\hat{\Phi }_{r} )$, the center of $\hat{\Phi }_{r} $, is trivial.
\end{proposition}

\begin{proof}
Denote by  $\hat{x}_{r,1} ,\ldots ,\hat{x}_{r,m} $ the standard generators of $\hat{\Phi }_{r} $ as profinite group, i.e. $\hat{x}_{r,i} =\left(\left(\begin{array}{cc} {x_{r-1,n,i} } & {{t}_{r-1,n,i} } \\ {0} & {1} \end{array}\right)\right)_{n=1}^{\infty } $ for $1\le i\le m$. We will show that $\bigcap _{i=1}^{m}Z_{\hat{\Phi }_{r} } \left(\hat{x}_{r,i} \right) =\left\{e\right\}$ (when, $Z_{\hat{\Phi }_{r} } \left(\hat{x}_{r,i} \right)$ is the center of $\hat{x}_{r,i} $ in $\hat{\Phi }_{r} $), and this actually means that: $Z(\hat{\Phi }_{r} )=\left\{e\right\}$.\newline

Let $\left(\left(\begin{array}{cc} {g_{n} } & {\sum _{i=1}^{m}a_{n,i} {t}_{r-1,n,i}  } \\ {0} & {1} \end{array}\right)\right)_{n=1}^{\infty } $ be an element of $ \bigcap _{i=1}^{m}Z_{\hat{\Phi }_{r} } \left(\hat{x}_{r,i} \right) $ such that for every $n'\in \mathbb{N} $, $\left(\begin{array}{cc} {g_{n'} } & {\sum _{i=1}^{m}a_{n',i} {t}_{r-1,n',i}  } \\ {0} & {1} \end{array}\right)\in \Phi _{r,n'} $. We will show that $a_{n',i'} =0$ and $g_{n'} =e$ for every $n'\in  \mathbb{N} $ and $1\le i'\le m$.\newline

Assume negatively that for the element $\left(\left(\begin{array}{cc} {g_{n} } & {\sum _{i=1}^{m}a_{n,i} {t}_{r-1,n,i}  } \\ {0} & {1} \end{array}\right)\right)_{n=1}^{\infty } $ there are $n'\in  \mathbb{N} $ and $1\le i'\le m$ such that $a_{n',i'}  \ne 0$. Consider $\Phi _{r,n'^{2} } $. Now, take any $j\ne i'$, and from the condition that $\left(\left(\begin{array}{cc} {g_{n} } & {\sum _{i=1}^{m}a_{n,i} {t}_{r-1,n,i}  } \\ {0} & {1} \end{array}\right)\right)_{n=1}^{\infty } \in Z_{\hat{\Phi }_{r} } \left(\hat{x}_{r,j} \right)$, we have:

\[\left(\begin{array}{cc} {g_{n'^{2} } } & {\sum_{i=1}^{m} a_{n'^{2} ,i} {t}_{r-1,n'^2,i}  } \\ {0} & {1} \end{array}\right)\left(\begin{array}{cc} {x_{r-1,n'^{2} ,j} } & {{t}_{r-1,n'^2,j} } \\ {0} & {1} \end{array}\right)=\]
\[=\left(\begin{array}{cc} {x_{r-1,n'^{2} ,j} } & {{t}_{r-1,n'^2,j} } \\ {0} & {1} \end{array}\right)\left(\begin{array}{cc} {g_{n'^{2} } } & {\sum_{i=1}^{m} a_{n'^{2} ,i} {t}_{r-1,n'^2,i}  } \\ {0} & {1} \end{array}\right).\]

After opening the last multiplication and comparing the coefficients of ${t}_{r-1,n'^2,i'} $ one can conclude that $a_{n'^{2} ,i'} =x_{r-1,n'^{2} ,j} a_{n'^{2} ,i'} $. As remembered, $a_{n'^{2} ,i'}\in \mathbb{Z} _{n'^{2} } [\Phi _{r-1,n'^{2} } ]$, and thus we can write $a_{n'^{2} ,i'} =\sum _{g\in \Phi _{r-1,n'^{2} } }\alpha _{g} g $ when $\alpha _{g} \in \mathbb{Z} _{n'^{2} } $ for every $g\in \Phi _{r-1,n'^{2} } $. Therefore, the equality $a_{n'^{2} ,i'} =x_{r-1,n'^{2} ,j} a_{n'^{2} ,i'} $, deduces that for every $g,h\in \Phi _{r-1,n'^{2} } $: $h\in \left\langle x_{r-1,n'^{2} ,j} \right\rangle g \Rightarrow \alpha _{g} =\alpha _{h}$. Now, according to proposition \ref{order}, $O\left(x_{r-1,n'^{2} ,j} \right)=\left(n'^{2} \right)^{r-1} $. Therefore, by denoting the representatives of the right cosets of $\left\langle x_{r-1,n'^{2} ,j} \right\rangle $ in the group $\Phi _{r-1,n'^{2} } $ in $\left\{g_{u} \right\}_{u\in U} $, one can write:

\noindent
\[a_{n'^{2} ,i'} =\sum _{g\in \Phi _{r-1,n'^{2} } }\alpha _{g} g =\sum _{u\in U}\alpha _{g_{u} } \left(\sum _{l=0}^{\left(n'^{2} \right)^{r-1} -1}\left(x_{r-1,n'^{2} ,j} \right)^{l}  \right) g_{u} .\]

Now, according to the inverse limit property, the natural map $\Phi _{r-1,n'^{2} } \to \Phi _{r-1,n'} $ induces a natural map $\mathbb{Z} _{n'^{2} } [\Phi _{r-1,n'^{2} } ]\to \mathbb{Z} _{n'} [\Phi _{r-1,n'} ]$, that under this map $a_{n'^{2} ,i'} \mapsto a_{n',i'} $. Considering that $O(x_{r-1,n',j} )=n'^{r-1} $, and that $a_{n',i'} \in \mathbb{Z} _{n'} [\Phi _{r-1,n'} ]$, we deduce that under this map:

\[\sum _{l=0}^{n'^{2\left(r-1\right)} -1}\left(x_{r-1,n'^{2} ,j} \right)^{l}  \mapsto \sum _{l=0}^{n'^{2\left(r-1\right)} -1}\left(x_{r-1,n',j} \right)^{l}  =n'^{r-1} \sum _{l=0}^{n'^{r-1} -1}\left(x_{r-1,n',j} \right)^{l}  =0.\]

Thus, $a_{n'^{2} ,i'} \mapsto 0$, and therefore $a_{n',i'} =0$.\newline

Now, after we proved that $a_{n',i'} =0$ for every $n'\in \mathbb{N} $ and for every $1\le i'\le m$, it is obvious that according to Theorem \ref{magnus}, also $g_{n'} =e$ for every $n'\in \mathbb{N} $.

\end{proof}

The following theorem is well known in more general cases ([G] Theorem 7.1, [Rob] Theorem 9.44). But one can notice that using the explicit description we gave for $\hat{\Phi} _{r} $, it is easy to give another proof for our specific case.

\begin{theorem}
\label{residual}
For every  $r\ge 0$, $\Phi _{r} $ is residually finite $p$ for all primes $p$.
\end{theorem}

\section{Methods from the solution of the CSP for $F_2$}

As mentioned at $\S0$, Bux, Ershov and Rapinchuk [BER] showed that $F_{2} $ has the CSP. In this section, we quote one of their results. We will use it in the opposite way: to show that $\Phi $, the free metabelian group on two generators, does \textbf{not} have the CSP.

\begin{proposition}[\text{[BER]}, Lemma 3.1]
Let $G$ be a finitely generated group. Then:
\begin{enumerate}
\item $G$ has the CSP if and only if the natural map $\hat{i}:\widehat{Aut(G)} \to Aut(\hat{G})$ is injective.
\item The map $\hat{i}$ induces $\hat{j}:\widehat{Out(G)} \to Out(\hat{G})$, and $\hat{j}$ is injective if and only if every subgroup of $Out(G)$ of finite index contains a subgroup of $Out(G)$ of the form $\Delta [K]:=\ker \left(OutG\to Out\left(G/K\right)\right)$, such that $K$ is a characteristic subgroup of $G$ of finite index.
\item Asumme $G$ is also residually finite, and that $\hat{G}$ is centerless. Then: $\hat{i}:\widehat{Aut(G)} \to Aut(\hat{G})$ is injective if and only if $\hat{j}:\widehat{Out(G)} \to Out(\hat{G})$ is injective.
\end{enumerate}
\end{proposition}

This proposition, together with propositions \ref{char}, \ref{comp},  \ref{center} and Theorem \ref{residual}, give:

\begin{lemma}
\label{conc}
Let $\Phi_r$ be the free solvable group on $m>1$ generators of solvability length $r>0$. Then, $\Phi_r$ has the CSP if and only if every subgroup of $Out(\Phi_r)$ of finite index contains a subgroup of the form $\ker (Out(\Phi_r) \to Out(\Phi_{r,n}))$ for some $n\in \mathbb{N}$.
\end{lemma}

\section{The CSP for free solvable groups on two generators}

As mentioned, Lemma \ref{conc} is valid for every $r,m>1$, but from now on we will concentrate on the case $m=2$. So, in this section, $\Phi_r$ denotes the free solvable group on \textbf{two} generators of solvability length $r$, and the generators of $\Phi_r$ will be denoted by $x_r$ and $y_r$. We will denote the generators of $\mathbb{Z}^2$ by $x$ and $y$.  \newline

Consider the natural map $\Phi_r \to \mathbb{Z} ^{2} $, defined by $x_r \mapsto x$, $y_r \mapsto y$. This map induces the natural maps: $Aut(\Phi_r) \to Aut(\mathbb{Z} ^{2}) $ and $Out(\Phi_r) \to Out(\mathbb{Z} ^{2}) $. It is easy to check that the map $Aut(\Phi_r) \to Aut(\mathbb{Z} ^{2}) $ is surjective, as $Aut(\mathbb{Z} ^{2}) $ is generated by the automorphisms:

\[ \alpha :\left\{\begin{array}{l} {x\mapsto xy} \\ {y\mapsto y} \end{array}\right. ,  \beta :\left\{\begin{array}{l} {x\mapsto x^{-1} } \\ {y\mapsto y} \end{array}\right.  , \gamma :\left\{\begin{array}{l} {x\mapsto y} \\ {y\mapsto x} \end{array}\right. \]

which are the images of the automorphisms:

\[ \alpha_r :\left\{\begin{array}{l} {x_r\mapsto x_r y_r} \\ {y_r\mapsto y_r} \end{array}\right. ,  \beta_r :\left\{\begin{array}{l} {x_r\mapsto x_r^{-1} } \\ {y_r\mapsto y_r} \end{array}\right.  , \gamma_r :\left\{\begin{array}{l} {x_r\mapsto y_r} \\ {y_r\mapsto x_r} \end{array}\right. \]

Therefore, the map $Out(\Phi_r) \to Out(\mathbb{Z} ^{2}) =Aut(\mathbb{Z} ^{2}) $ is also surjective. We want to show that this map is also injective, and to conclude that: $Out(\Phi_r) \cong Out(\mathbb{Z} ^{2}) =Aut(\mathbb{Z} ^{2}) \cong GL_{2} (\mathbb{Z} )$.\newline

\begin{definition}
Let $G$ be a group and $IA(G):=\ker(Aut(G) \to Aut(G/G'))$. An element of $IA(G)$ is called an "IA-automorphism".
\end{definition}

The sequence $1\to IA(\Phi_r )\to Aut(\Phi_r) \to Aut(\mathbb{Z} ^{2}) \to 1$ is an exact sequence, and so is the sequence $1\to Inn(\Phi_r) \to Aut(\Phi_r) \to Out(\Phi_r) \to 1$. Clearly, $Inn(G)\subseteq IA(G)$, and in particular, $Inn(\Phi_r) \subseteq IA(\Phi_r )$. Bachmuth, Formanek and Mochizuki, gave few conditions for equality:\newline

\begin{theorem} [\text{[BFM]}]
Let $F$ be the free group on two generators, and let $R$ be a normal subgroup of $F$ satisfying:
\begin{enumerate}
\item[-] $R \leqslant F'$
\item[-] $\mathbb{Z}(F/R)$, the integral group ring of $F/R$, is a domain.
\end{enumerate}
Then, $IA(F/R') = Inn(F/R')$.
\end{theorem}

Let us now mention:

\begin{theorem} [\text{Kropholler-Linnell-Moody [KLM]}]
Let $k$ be a division ring and let $G$ be a solvable by finite group. If $G$ is torison-free, then $k[G]$ is a domain.
\end{theorem}

\begin{corollary}
For every $r \in \mathbb{N}$:
\begin{enumerate}
\item $ IA(\Phi_r)= Inn(\Phi_r)$.
\item The map $Out(\Phi_r) \to GL_{2} (\mathbb{Z}) $ is an isomorphism.
\end{enumerate}
\end{corollary}

\begin{conjecture}
\label{GL}
With the above isomorphism: for $r,n \in \mathbb{N}$: \[\ker(Out(\Phi_r) \to Out(\Phi_{r,n}))=\ker(GL_{2} (\mathbb{Z} ) \to GL_{2} (\mathbb{Z}_{n^r}) ) \]
\end{conjecture}

As $\mathbb{Z}^2$ does not have the CSP, there is a subgroup of $Aut(\mathbb{Z}^2 ) \cong GL_2 (\mathbb{Z})$ of finite index containing no subgroup of the form $\ker (GL_2 (\mathbb{Z}) \to GL_2 (\mathbb{Z}_n))$. Therefore, if the conjecture is true, the same subgroup of $GL_2 (\mathbb{Z})$ does not contain any subgroup of the form $ \ker(Out(\Phi_r) \to Out(\Phi_{r,n}))$, and thus, by Lemma \ref{conc}, $\Phi_r$ does not have the CSP.\newline

We consider now the following notations:
\begin{definition}
Define:
\begin{enumerate}
\item [-] $GL'_{2}(\mathbb{Z}_{n})=Im(GL_{2}(\mathbb{Z})\rightarrow GL_{2}(\mathbb{Z}_{n}))$.
\item [-] $Aut'(\Phi_{r,n})=Im(Aut(\Phi_r)\rightarrow Aut(\Phi_{r,n}))$.
\item [-] $Out'(\Phi_{r,n})=Im(Out(\Phi_r)\rightarrow Out(\Phi_{r,n}))$.
\item [-] $IA'(\Phi_{r,n})=IA(\Phi_{r,n})\cap Aut'(\Phi_{r,n})$.
\end{enumerate}
\end{definition}

With the above notations we have:

\begin{lemma}
Conjecture \ref{GL} is true iff $IA'(\Phi_{r,n})=Inn(\Phi_{r,n})$.
\end{lemma}

\begin{proof}
As was shown in proposition \ref{ab} $\Phi_{r,n}^{ab} = \Phi_{r,n}/[\Phi_{r,n},\Phi_{r,n}] \backsimeq \mathbb{Z}_{n^r}$. Thus, the natural surjective map $Aut(\Phi_r)\rightarrow GL_{2}(\mathbb{Z})$ induces a natural surjective map $Aut'(\Phi_{r,n})\rightarrow GL'_{2}(\mathbb{Z}_{n^r})$ whose kernel is exactly $IA'(\Phi_{r,n})$. This map induces another natural surjective map $Out'(\Phi_{r,n})\rightarrow GL'_{2}(\mathbb{Z}_{n^r})$, and this map is an isomorphism iff $IA'(\Phi_{r,n})=Inn(\Phi_{r,n})$. On the other hand, the latter is an isomorphism iff conjecture \ref{GL} is true. So the lemma follows.

\end{proof}

\begin{theorem}
\label{last}
For $r=2$ we have $IA'(\Phi_{2,n})=Inn(\Phi_{2,n})$  for every $n \in \mathbb{N}$.
\end{theorem}

Thus, once we prove Theorem \ref{last}, Theorem \ref{main} is also proven.

\section{Proof of Theorem \ref{last}}

In this section $\Phi=\Phi_0=\Phi_{2,2}$ denotes the free metabelian group on two generators. We also write $\Phi_n$ instead of $\Phi_{2,2,n}$. Let $n \in \mathbb{N} \cup \{0\}$ and let $x_n,y_n$ be the generators of $(\mathbb{Z}_n)^2$. By Theorem \ref{magnus} we can consider $\Phi_n $ as the group generated by $x_n^{*}=\left(\begin{array}{cc} {x_n} & {t_{x_n} } \\ {0} & {1} \end{array}\right)$ and $y_n^{*}=\left(\begin{array}{cc} {y_n} & {t_{y_n} } \\ {0} & {1} \end{array}\right)$, and the elements of $\Phi_n$ are the matrices of the form $\left(\begin{array}{cc} {x_n^{i} y_n^{j} } & {a_{1} (x_n,y_n)t_{x_n} +a_{2} (x_n,y_n)t_{y_n} } \\ {0} & {1} \end{array}\right)$, such that:

\begin{enumerate}
\item $ i,j\in \mathbb{Z}$
\item $a_{1} (x_n,y_n),a_{2} (x_n,y_n)\in \mathbb{ Z}_n [\mathbb{ Z}_n ^{2} ]=\mathbb{ Z}_n [x_n^{\pm 1} ,y_n^{\pm 1} ]$
\item $1-x_n^{i} y_n^{j} = a_{1} (x_n,y_n)\cdot (1-x_n)+a_{2} (x_n,y_n)\cdot (1-y_n) $
\end{enumerate}

Every automorphism of $\Phi_n $ is determined by its action on the two generators: $\left(\begin{array}{cc} {x_n} & {t_{x_n} } \\ {0} & {1} \end{array}\right)$ and $\left(\begin{array}{cc} {y_n} & {t_{y_n} } \\ {0} & {1} \end{array}\right)$. Therefore, when we want to describe an automorphism, $\sigma $, of $\Phi _n$ we can write:

\[\sigma  =\left\{\begin{array}{l} {\left(\begin{array}{cc} {x_n} & {t_{x_n} } \\ {0} & {1} \end{array}\right)\mapsto \left(\begin{array}{cc} {x_n^{i} y_n^{j} } & {a_{1} t_{x_n} +a_{2} t_{y_n} } \\ {0} & {1} \end{array}\right)} \\ {\left(\begin{array}{cc} {y_n} & {t_{y_n} } \\ {0} & {1} \end{array}\right)\mapsto \left(\begin{array}{cc} {x_n^{k} y_n^{l} } & {b_{1} t_{x_n} +b_{2} t_{y_n} } \\ {0} & {1} \end{array}\right)} \end{array}\right. .\]

\begin{definition}
Let $n \in \mathbb{N} \cup \{0\}$. For $\sigma  \in Aut(\Phi_n)$ define: $d(\sigma ):=\det \left(\begin{array}{cc} {a_{1} } & {a_{2} } \\ {b_{1} } & {b_{2} } \end{array}\right)$. Notice that  $d(\sigma  )$ is well defined, as $a_{1} ,a_{2} ,b_{1} ,b_{2} \in \mathbb{Z}_n [\mathbb{Z}_n ^{2} ]$ and $\mathbb{Z} _n[\mathbb{Z}_n ^{2} ]$ is a commutative ring.
\end{definition}

\begin{lemma}
\label{det}
For $\sigma  \in Aut(\Phi_n)$, $d(\sigma )$ is invertible in $\mathbb{Z}_n [\mathbb{Z}_n ^{2} ]$.
\end{lemma}

\begin{proof}
As $\sigma $ is an automorphism, we can write $\left(\begin{array}{cc} {x_n} & {t_{x_n} } \\ {0} & {1} \end{array}\right)$ and $\left(\begin{array}{cc} {y_n} & {t_{y_n} } \\ {0} & {1} \end{array}\right)$ as words with the matrices: $\left(\begin{array}{cc} {x_n^{i} y_n^{j} } & {a_{1} t_{x_n} +a_{2} t_{y_n} } \\ {0} & {1} \end{array}\right)$ and $\left(\begin{array}{cc} {x_n^{k} y_n^{l} } & {b_{1} t_{x_n} +b_{2} t_{y_n} } \\ {0} & {1} \end{array}\right)$. Therefore, it is easy to check, by induction on the length of these words, that there are polynomials $c_{1} ,c_{2} ,d_{1} ,d_{2} \in {\mathbb{Z}_n} [\mathbb{Z}_n ^{2} ]$ such that:
\[\left(\begin{array}{cc} {x_n} & {t_{x_n} } \\ {0} & {1} \end{array}\right)=\left(\begin{array}{cc} {x_n} & {(c_{1} a_{1} +d_{1} b_{1} )t_{x_n} +(c_{1} a_{2} +d_{1} b_{2} )t_{y_n} } \\ {0} & {1} \end{array}\right)\Rightarrow _{} c_{1} a_{1} +d_{1} b_{1} =1,_{} c_{1} a_{2} +d_{1} b_{2} =0\]
\[\left(\begin{array}{cc} {y_n} & {t_{y_n} } \\ {0} & {1} \end{array}\right)=\left(\begin{array}{cc} {y_n} & {(c_{2} a_{1} +d_{2} b_{1} )t_{x_n} +(c_{2} a_{2} +d_{2} b_{2} )t_{y_n} } \\ {0} & {1} \end{array}\right)\Rightarrow _{} c_{2} a_{1} +d_{2} b_{1} =0,_{} c_{2} a_{2} +d_{2} b_{2} =1.\]

I.e., we got that: $\left(\begin{array}{cc} {c_{1} } & {d_{1} } \\ {c_{2} } & {d_{2} } \end{array}\right)\cdot \left(\begin{array}{cc} {a_{1} } & {a_{2} } \\ {b_{1} } & {b_{2} } \end{array}\right)=I$, and therefore, $\det \left(\begin{array}{cc} {a_{1} } & {a_{2} } \\ {b_{1} } & {b_{2} } \end{array}\right)$ is invertible in $\mathbb{Z}_n [\mathbb{Z}_n ^{2} ]$.

\end{proof}

From the last lemma it follows that if $\sigma \in Aut(\Phi)$ then $d(\sigma  )=\pm x^{r} y^{s} $ for some $r,s\in \mathbb{Z} $, as these are the invertible elements of $\mathbb{Z} [\mathbb{Z} ^{2} ]$ ([CF], chapter 8).

\begin{lemma}
\label{comsub}
Let $n \in \mathbb{N} \cup \{0\}$. For the commutator subgroup of $\Phi_n$ we have:
\begin{equation}
\label{com}
[\Phi_n,\Phi_n]=\left\{ \left(\begin{array}{cc}
1 & (1-y_n)\cdot p \cdot t_{x_n}-(1-x_n)\cdot p\cdot t_{y_n}\\
0 & 1
\end{array}\right)\,\mid\, p\in\mathbb{Z}_{n}[\mathbb{Z}_{n}^{2}]\right\}
 \end{equation}
 \end{lemma}

\begin{proof}
Let us write $h^{*}=\left(\begin{array}{cc} h & b_{1}t_{x_n}+b_{2}t_{y_n}\\0 & 1\end{array}\right) $ and $g^{*}=\left(\begin{array}{cc}g & a_{1}t_{x_n}+a_{2}t_{y_n}\\0 & 1\end{array}\right)$ for two arbitrary elements in $\Phi_n$. Then, by direct computation and by using the identities:
\[1-g=a_{1}(1-x_n)+a_{2}(1-y_n)\qquad\text{,}\qquad1-h=b_{1}(1-x_n)+b_{2}(1-y_n)
 \]
 we deduce:
\[[g^{*},h^{*}]=ghg^{-1}h^{-1}=\left(\begin{array}{cc}
1 & (1-y_n)(a_{1}b_{2}-b_{1}a_{2})t_{x_n}-(1-x_n)(a_{1}b_{2}-b_{1}a_{2})t_{y_n}\\
0 & 1
\end{array}\right)
 \]
This shows that multiplications of commutators of  $\Phi_n$ are of the form \ref{com}.\newline

For the other direction, one should notice that by the notations $x_n^{*}=\left(\begin{array}{cc}
x_n & t_{x_n}\\
0 & 1
\end{array}\right)$ and $y_n^{*}=\left(\begin{array}{cc}
y_n & t_{y_n}\\
0 & 1
\end{array}\right)$, we have:
\[[x_n^{*},y_n^{*}]=\left(\begin{array}{cc}
1 & (1-y_n)t_{x_n}-(1-x_n)t_{y_n}\\
0 & 1
\end{array}\right)
 \]
Moreover, conjugation by $x_n^{*}$ or $y_n^{*}$ multiply the right upper coordinate by $x_n$ or $y_n$ respectively. That shows that we can reach every $p\in\mathbb{Z}_{n}[\mathbb{Z}_{n}^{2}] $.

 \end{proof}

We can now prove Theorem \ref{last}:

\begin{proof}
It is obvious that $IA'(\Phi_n)\supseteq Inn(\Phi_n)$. On the other hand, consider an element $\sigma\in IA'(\Phi_{n})$. As $\sigma\in IA(\Phi_{n})$, by Lemma \ref{comsub}, we can describe $\sigma$ as follows:

\[\sigma=\begin{cases}
\left(\begin{array}{cc}
x_n & t_{x_n}\\
0 & 1
\end{array}\right)\mapsto & \left(\begin{array}{cc}
x_n & t_{x_n}\\
0 & 1
\end{array}\right)\left(\begin{array}{cc}
1 & (1-y_n)(x_n^{-1}p)t_{x_n}-(1-x_n)(x_n^{-1}p)t_{y_n}\\
0 & 1
\end{array}\right)=\\
 & =\left(\begin{array}{cc}
x_n & [1+(1-y_n)p]t_{x_n}+-(1-x_n)pt_{y_n}\\
0 & 1
\end{array}\right)\\
\left(\begin{array}{cc}
y_n & t_{y_n}\\
0 & 1
\end{array}\right)\mapsto & \left(\begin{array}{cc}
y_n & t_{y_n}\\
0 & 1
\end{array}\right)\left(\begin{array}{cc}
1 & (1-y_n)(y_n^{-1}q)t_{x_n}-(1-x_n)(y_n^{-1}q)t_{y_n}\\
0 & 1
\end{array}\right)=\\
 & =\left(\begin{array}{cc}
y_n & (1-y_n)qt_{x_n}+[1-(1-x_n)q]t_{y_n}\\
0 & 1
\end{array}\right)
\end{cases}\]

for some $p,q\in\mathbb{Z}_{n}[\mathbb{Z}_{n}^{2}]$.\newline

Now, one can compute the determinant of $\sigma$ and deduce that:

\[
d(\sigma)=[1+(1-y_n)p]\cdot[1-(1-x_n)q]+(1-x_n)p\cdot(1-y_n)q=1+(1-y_n)p-(1-x_n)q
 \]

But, by considering that $\sigma\in Aut'\Phi_{n}$ and  by applying the conclution of Lemma \ref{det}, we conclude that there are $i,j\in\mathbb{Z}_{n}$ such that:
\[1+(1-y_n)p-(1-x_n)q=\pm x_n^{i}y_n^{j}\]

 Moreover, one should notice that if $x_n\mapsto1,\; y_n\mapsto1$ then:
\[1+(1-y_n)p-(1-x_n)q\mapsto1\]
and therefore:
\[1+(1-y_n)p-(1-x_n)q=x_n^{i}y_n^{j}\]
Thus:
\begin{equation}
\label{ma}
1-x_n^{i}y_n^{j}=(1-x_n)q-(1-y_n)p
\end{equation}

Now, according to  Theorem \ref{magnus}, we conclude that the element:
\[\left(\begin{array}{cc}
x_n^{i}y_n^{j} & qt_{x_n}-pt_{y_n}\\
0 & 1
\end{array}\right)\]

is in the image of $\Phi_{n}$ in $R_n (\mathbb{Z}_n)$. Direct computation and use of equation \ref{ma} shows that:
\[\sigma=Inn\left(\begin{array}{cc}
x_n^{i}y_n^{j} & qt_{x_n}-pt_{y_n}\\
0 & 1
\end{array}\right)\]

so $\sigma \in Inn(\Phi_n)$ as required.

\end{proof}

That completes the main result of the paper. Now, one can ask whether $IA(\Phi_n)=Inn(\Phi_n)$. The next  explicit example  shows that this is not true in general.

Let us consider the map:
\[\sigma=\begin{cases}
x_n^{*}\rightarrow & y_n^{*}x_n^{*}y_n^{*-1}=x_n^{*}[x_n^{*-1},y_n^*]=\left(\begin{array}{cc}
x_n & y_n t_{x_n}+(1-x_n)t_{y_n}\\
0 & 1
\end{array}\right)\\
y_n^{*}\rightarrow & x_n^{*}y_n^{*}x_n^{*-1}=y_n^{*}[y_n^{*-1},x_n^*]=\left(\begin{array}{cc}
y_n & (1-y_n)t_{x_n}+x_n t_{y_n}\\
0 & 1
\end{array}\right)
\end{cases}
\]

One should notice that $\sigma$ defines a homomorphism of $\Phi_n$, as the relations that define $\Phi_n$ are sutisfied by all the elements of $\Phi_n$, and in particular, by $\sigma(x_n^{*})$ and $\sigma(y_n^{*})$. We are now going to show that when $n=p$ is prime, $\sigma$ defines an IA-automorphism of $\Phi_n$ which is not inner.\newline

\begin{lemma}
When $p$ is prime, the polynomial $x_p+y_p-1$ is invertible in $\mathbb{Z}_{p}[\mathbb{Z}_{p}^{2}]$.
\end{lemma}

\begin{proof}
When $p$ is prime, we have:
\[(x_p+y_p-1)\cdot\sum_{i=0}^{p-1}(x_p+y_p)^{i}	=	(x_p+y_p)^{p}-1=\]
\[
	=	x_p^{p}+\left(\begin{array}{c}
p\\
1
\end{array}\right)x_p^{p-1}y_p+\ldots+\left(\begin{array}{c}
p\\
p-1
\end{array}\right)x_p y_p^{p-1}+y_p^{p}-1=\]
\[
	=	1+0+\ldots+0+1-1=1
 \]

\end{proof}

\begin{proposition}
When $p$ is prime, $\sigma$ is onto. In particular, $\sigma$ is an isomorphism.
\end{proposition}

\begin{proof}
A direct computation shows that:
\[[\sigma(x_p^{*}),\sigma(y_p^{*})]=\left(\begin{array}{cc}
1 & (1-y_p)(x_p+y_p-1)t_{x_p}-(1-x_p)(x_p+y_p-1)t_{y_p}\\
0 & 1
\end{array}\right)
 \]
 Moreover, direct computation shows that conjugation by $\sigma(x_p^{*})$ and $\sigma(y_p^{*})$ multiply the right upper coordinate by $x_p$ and $y_p$ repectively. By the previous lemma $x_p+y_p-1$ is invertible in $\mathbb{Z}_{p}[\mathbb{Z}_{p}^{2}]$. Now, putting all together and by considering Lemma \ref{comsub} we conclude that $\sigma$ reaches all $[\Phi_p,\Phi_p]$.\newline

On the other hand, the projection of $\sigma$ to $\Phi_p^{ab}=\Phi_p/[\Phi_p,\Phi_p]$ is the identity map, so $\sigma$ reaches also all $\Phi_p^{ab}$. Therefore, $\sigma$ reaches all $\Phi_p$, as required.\newline

As $\Phi_p$ is a finite group we deduce that $\sigma$ is an isomorphism.

\end{proof}

\begin{proposition}
When $n=p$ is prime, the automorphism $\sigma$ is an IA-automorphism which is not inner.
\end{proposition}

\begin{proof}
In the proof of the previous proposition we saw that $\sigma$ is an IA-automorphism. On the other hand, assume that  $\sigma$ is inner. By direct computation we deduce that the determinant of $\sigma$ equals $d(\sigma)=x_p+y_p-1$. On the other hand, as $Inn(\Phi_p)\subseteq Aut'(\Phi_p)$, we conclude that $d(\sigma)=\pm x_p^{r}y_p^{s}$, and this is a contradiction.

 \end{proof}

\end{document}